\documentclass[12pt]{amsart}
\usepackage{amsfonts}
\usepackage{amssymb}
\usepackage{graphicx}
\usepackage{pgf,tikz}
\usetikzlibrary{arrows}
\usepackage[letterpaper, left=2.5cm, right=2.5cm, top=2.5cm,
bottom=2.5cm,dvips]{geometry}
\usepackage[T1]{fontenc} 
\usepackage[backref]{hyperref}
\hypersetup{
	colorlinks,
	linkcolor={blue!60!black},
	citecolor={green!60!black},
	urlcolor={red!60!black}
}

\newtheorem{theorem}{Theorem}
\theoremstyle{plain}

\newtheorem{conjecture}{Conjecture}
\newtheorem{corollary}{Corollary}

\newtheorem{lemma}{Lemma}

\newtheorem{proposition}{Proposition}
\newtheorem{remark*}{Remark}

\DeclareMathOperator{\ch}{ch}
\DeclareMathOperator{\p}{p}
\DeclareMathOperator{\AT}{AT}
\numberwithin{equation}{section}

\begin{document}
\title[Graph polynomials and paintability of plane graphs]{Graph polynomials and paintability of plane graphs}

\author{Jaros\l aw Grytczuk}
\address{Faculty of Mathematics and Information Science, Warsaw University
	of Technology, 00-662 Warsaw, Poland}
\email{j.grytczuk@mini.pw.edu.pl}

\author{Stanislav Jendrol'}
\address{Institute of Mathematics, Faculty of Science, P.J. \v {S}af\'{a}rik University in Ko\v {s}ice, Slovakia}
\email{stanislav.jendrol@upjs.sk}

\author{Mariusz Zaj\k {a}c}
\address{Faculty of Mathematics and Information Science, Warsaw University
	of Technology, 00-662 Warsaw, Poland}
\email{m.zajac@mini.pw.edu.pl}

\thanks{Supported by the Polish National Science Center, Grant Number: NCN 2017/26/D/ST6/00264, and by the Slovak Research and Development Agency under the contract No: APVV-19-0153.}

\begin{abstract} There exists a variety of coloring problems for plane graphs, involving vertices, edges, and faces in all possible combinations. For instance, in the \emph{entire coloring} of a plane graph we are to color these three sets so that any pair of adjacent or incident elements get different colors. We study here some problems of this type from algebraic perspective, focusing on the \emph{facial} variant. We obtain several results concerning the \emph{Alon-Tarsi number} of various graphs derived from plane embeddings. This allows for extensions of some previous results for \emph{choosability} of these graphs to the game theoretic variant, know as \emph{paintability}. For instance, we prove that every plane graph is facially entirely \emph{$8$-paintable}, which means (metaphorically) that even a color-blind person can facially color the entire graph form lists of size $8$. 
\end{abstract}
\maketitle

\section{Introduction}
There exist a variety of coloring problems involving diverse combinations of vertices, edges, and faces of plane graphs. For instance, in 1965 Ringel \cite{Ringel} proved that \emph{six} colors are sufficient to color the vertices and the faces of any plane triangulation so that any pair of adjacent or incident elements get different colors. He also asked if this is true for general plane graphs and the affirmative answer was provided by Borodin \cite{Borodin Ringel} (see \cite{Borodin Survey}). If we additionally introduce edges to the game, then there is no finite bound on the number of colors since plane graphs may have vertices with arbitrarily large degree. However, if we restrict to \emph{facially adjacent} edges, then it can be proved that \emph{eight} colors are already sufficient, as demonstrated by Fabrici, Jendrol', and Voigt in \cite{FabriciJV} (see \cite{CzapJendrol}).

In order to conveniently discuss problems of the above type we introduce the following notation. Let $G$ be a \emph{plane graph}, that is, a fixed embedding of a planar graph into the plane with no pairs of crossing edges. We shall assume throughout the paper that our plane graphs are \emph{simple} (no loops and multiple edges), and moreover, all faces are bounded by simple cycles. By $V,E$, and $F$ we denote traditionally the set of vertices, edges, and faces of $G$, respectively.

Let $G_v$ denote the planar graph whose plane embedding is $G$. Also, let $G_f$ denote the graph whose vertices are faces of $G$ with two vertices $x,y\in F$ adjacent in $G_f$ when they share an edge in $G$. Furthermore, denote by $B_{vf}$ the bipartite graph whose bipartition classes are $V$ and $F$ with $x\in V$ adjacent to $y\in F$ when $x$ is incident to $y$ in $G$. Finally, let us define a graph $G_{vf}$ by $G_{vf}=G_v\cup G_f\cup B_{vf}$. Then the Ringel-Borodin theorem mentioned above can be shortly expressed as:
$$\chi(G_{vf})\leqslant 6.$$

Let $G_e$ denote the line graph of $G_v$. Similarly we may define bipartite graphs $B_{ve}$ and $B_{ef}$ on bipartition classes $V$ and $E$, and $E$ and $F$, respectively, with adjacency relation induced by the incidence relation between the corresponding elements in the plane graph $G$. Also, in analogy to the graph $G_{vf}$, we may similarly define graphs $G_{ve}$, $G_{ef}$, or even $G_{vef}$, where the last is the \emph{entire graph} of $G$: $$G_{vef}=G_v\cup G_e\cup G_f\cup B_{ve}\cup B_{vf}\cup B_{ef}.$$

In the present paper we consider coloring of the above graphs from algebraic perspective. We shall however concentrate on the \emph{facial} variant. This means that instead of the whole line graph $G_e$ we consider only its subgraph $\overline{G}_e$ restricted to \emph{facially adjacent} edges, that is, pairs of edges sharing a vertex and a face of $G$. Notice that this definition coincides with the traditional one (based on facial walks) for graphs we considered here (all faces are simple cycles). With this restriction we consider facial analogs of previously defined graphs which we denote as $\overline{G}_{ve}, \overline{G}_{ef}$, and $\overline{G}_{vef}$. For instance, $\overline{G}_{ve}=G_v\cup \overline{G}_e\cup B_{ve}$. In particular, we may state the above mentioned result of Fabrici, Jendrol', and Voigt as: $$\chi (\overline{G}_{vef})\leqslant 8.$$

Recall that the \emph{choice number} of a graph $G$, denoted by $\ch (G)$, is the least integer $k$ such that the vertices of $G$ can be properly colored even if the color of each vertex $v$ is restricted to an arbitrary \emph{list} $L(v)$ of $k$ colors assigned to $v$. Thomassen \cite{Thomassen 5} proved that every planar graph $G$ satisfies $\ch (G)\leqslant 5$, while Voigt \cite{Voigt} found a non-$4$-choosable example. Wang and Lih \cite{WangLih} proved that every plane graph $G$ satisfies $\ch (G_{vf})\leqslant 7$, but it is not known if this bound is optimal.

Given a graph $G$, let $P_G$ be the \emph{graph polynomial} of $G$ defined by
\begin{equation}
P_G=\prod_{xy\in E}(x-y).
\end{equation}
We consider $P_G$ as a polynomial over the field $\mathbb{R}$ of real numbers. The \emph{Alon-Tarsi number} of $G$, denoted by $\AT (G)$, is the least integer $k$ such that $P_G$ contains a non-vanishing monomial whose degree in each variable is at most $k-1$. The celebrated Combinatorial Nullstellensatz of Alon \cite{AlonCN} (see also \cite{AlonTarsi}) implies that every graph $G$ satisfies
$$\ch (G)\leqslant \AT(G).$$

Schauz \cite{Schauz} extended this result to the game theoretic variant of the choice number $\ch(G)$, known as the \emph{painting number} $\p(G)$. It was introduced independently by Schauz \cite{Schauz Paint Correct} and Zhu \cite{Zhu Online}. We postpone a precise definition of $\p (G)$ to the penultimate section, where we also give a refined proof of Schauz's theorem. Roughly speaking, a graph $G$ satisfies $\p (G)\leqslant k$, or is \emph{$k$-paintable}, if a color-blind \emph{Painter} can color the vertices of $G$ from any lists of size $k$ revealed sequentially by \emph{Lister}. In each round of the play Lister picks a set of vertices $S$ containing a certain color in their lists, and then Painter colors a subset of $S$ properly by this color. It is clear that every graph satisfies $\ch (G)\leqslant \p (G)$, but, as proved by Schauz \cite{Schauz}, it also holds that $$\p (G)\leqslant \AT (G).$$

We obtain several results on the Alon-Tarsi number of various combined plane graphs, extending thereby many previous results concerning facial choosability (see \cite{FabriciJV}). For instance, we prove that for every plane graph $G$ we have $$\AT (\overline{G}_{vef})\leqslant 8,$$which implies that every plane graph is facially entirely $8$-paintable.
 
\section{The polynomial method and graph choosability}
In this section we state some basic notions and results that we will use later on.

\subsection{Graph polynomials}

Let $G$ be a simple graph on the set of vertices $V=\{x_1,x_2,\dots,x_n\}$. Let $P_G$ be the \emph{graph polynomial} of $G$, defined by
\begin{equation}\label{Graph Polynomial}
P_G=\prod_{x_ix_j\in E}(x_i-x_j).
\end{equation}
We identify symbols denoting vertices of $G$ with variables of $P_G$. We will consider $P_G$ as a polynomial over the field $\Bbb R$ of real numbers. Notice that $P_G$ is defined uniquely up to the sign, depending on the choice between the two possible expressions, $(x_i-x_j)$ or $(x_j-x_i)$, representing the edge $x_ix_j$ in the product (\ref{Graph Polynomial}). We shall assume that this choice is fixed. Notice also that $P_G$ is a \emph{uniform} polynomial, which means that all monomials in $P_G$ have the same total degree (equal to the number of edges of $G$).

In the process of expanding the polynomial $P_G$, one creates monomials by picking one variable from each factor $(x_i-x_j)$. Every such monomial corresponds to the unique orientation of the edges of $G$ obtained by directing the edge $x_ix_j$ \emph{towards} the picked variable. In this way the degrees of variables in the monomial coincide with the in-degrees of the vertices in the corresponding orientation.

Let $\mathcal{M}_G$ denote the multi-set of all monomials arising in this way. So, the cardinality of $\mathcal{M}_G$ is equal to $2^m$, where $m=|E(G)|$, and the multipicity of each monomial $M$ is equal to the number of orientations of $G$ sharing the same in-degree sequence (corresponding to the degrees of variables in $M$). The \emph{sign} of a monomial $M\in \mathcal{M}_G$ is the product of signs of all variables picked to form $M$. The \emph{coefficient} of a monomial $M$ in $P_G$, denoted as $c_P(M)$, is the sum of signs of all copies of $M$ in $\mathcal{M}_G$. A monomial $M$ is called \emph{non-vanishing} in $P_G$ if $c_P(M)\neq 0$.

\subsection{Combinatorial Nullstellensatz}

For a monomial $M$, let $\deg_{x_i}(M)$ denote the degree of the variable $x_i$ in $M$. The \emph{Alon-Tarsi number} of a graph $G$, denoted by $\AT (G)$, is the least integer $k$ such that there exists a non-vanishing monomial $M$ in $P_G$ satisfying $\deg_{x_i}(M)\leqslant k-1$ for every variable $x$.

The following famous result of Alon \cite{AlonCN} links graph polynomials to graph choosability.

\begin{theorem}[Combinatorial Nullstellensatz \cite{AlonCN}] Let $P$ be a polynomial in $\mathbb F[x_1,x_2,\dots,x_m]$ over any field $\mathbb F$. Suppose that there is a non-vanishing monomial $x_1^{k_1} x_2^{k_2}\cdots x_m^{k_m}$ in $P$ whose degree is equal to the degree of $P$. Then, for arbitrary sets $A_i\subseteq \mathbb F$, with $|A_i|=k_i+1$, there is a choice of elements $a_i\in A_i$ such that $P(a_1,a_2,\dots,a_m)\neq 0$.
\end{theorem}

Since $P_G$ is uniform, we get immediately that every graph $G$ satisfies $$\ch (G)\leqslant \AT(G).$$

We will make a frequent use of the following simple observation. Let $G$ be a graph whose edges are split into two disjoint subsets:
$$E(G)=A\cup B.$$
Then the graph polynomial $P_G$ can be written as the product of two corresponding polynomials:
$$P_G=\prod_{x_ix_j\in E(G)}(x_i-x_j)=\prod_{x_ix_j\in A}(x_i-x_j)\prod_{x_ix_j\in B}(x_i-x_j)=P_AP_B.$$Suppose now, that $M_A$ and $M_B$ are non-vanishing monomials with the least possible degrees in the polynomials $P_A$ and $P_B$, respectively. Then the monomial $M=M_AM_B$ appears in the multiset of monomials $\mathcal{M}_G$. Moreover, the maximum degree of a variable in $M$ is the sum of maximum degrees of variables in $M_A$ and $M_B$. So, if one could only prove that $M$ is non-vanishing, then the following bound would follow:
$$\AT (G)\leqslant \AT (A)+\AT (B) -1.$$
Therefore, to establish an upper bound on the Alon-Tarsi number of $G$ one may try to split it into a union of two subgraphs with low Alon-Tarsi numbers, and try to prove that the monomial $M=M_AM_B$ is non-vanishing.

\subsection{Bipartite graphs}
We start with the following simple result for graph polynomials of bipartite graphs.

\begin{proposition}\label{Proposition Bipartite Polynomials}
	Let $G$ be a bipartite graph with bipartition classes $X$ and $Y$. Then every monomial $M\in \mathcal{M}_G$ is non-vanishing. 
\end{proposition}

\begin{proof}
	Let $X=\{x_1,x_2,\dots,x_r\}$, $Y=\{y_1,y_2,\dots,y_s\}$, and let
	\begin{equation}
	P_G=\prod_{x_iy_j\in E(G)}(x_i-y_j).
	\end{equation}
	Any monomial $M\in\mathcal{M}_G$ has the form:
	$$M=x_1^{a_1}x_2^{a_2}\cdots x_r^{a_r}y_1^{b_1}y_2^{b_2}\cdots y_s^{b_s},$$
	with $a_i\geqslant0$ and $b_j\geqslant0$ for all $i=1,2,\dots,r$ and $j=1,2,\dots, s$.
	Notice that in forming monomials for $P_G$, the variables form $X$ are always positive, while the variables form $Y$ are always negative. It follows that the sign of $M$ is equal to $(-1)^{b_1+\cdots +b_s}$ and is therefore the same for each copy of $M$ in $\mathcal{M}_G$. So, $M$ is non-vanishing (over $\mathbb R$), as asserted.
\end{proof}

The above proposition gives immediately the following well-known result of Alon and Tarsi \cite{AlonTarsi}.
\begin{corollary}[Alon and Tarsi \cite{AlonTarsi}]\label{Corollary AT Bipartite}
	Let $G$ be a bipartite graph and let $d$ be the least integer such that $G$ has an orientation with all in-degrees at most $d$. Then $\AT (G)= d+1$. In particular, every planar bipartite graph $G$ satisfies $\AT (G)\leqslant 3$.
\end{corollary}

\subsection{Planar graphs}
It is not hard to demonstrate that every planar graph $G$ satisfies $\AT (G)\leqslant 6$. This follows from the following more general fact.

\begin{proposition}\label{Proposition AT Degenerate}
Let $G$ be a graph on $n$ vertices having an acyclic orientation with in-degree sequence $(d_1,d_2,\dots, d_n)$. Then the monomial $M=x_1^{d_1}x_2^{d_2}\cdots x_n^{d_n}$ is non-vanishing in $P_G$, and its coefficient satisfies $c_{P_G}(M)=\pm1$.
\end{proposition}

\begin{proof}Every graph has only one acyclic orientation with fixed in-degree sequence. So, the monomial $M$ corresponding to this orientation is unique in the multiset $\mathcal{M}_G$ and therefore non-vanishing.
	\end{proof}

Since every planar graph $G$ has an acyclic orientation with maximum in-degree at most $5$, it follows that $\AT (G)\leqslant 6$. The following recent result of Zhu \cite{ZhuAT} is a far reaching strengthening of this observation. It also extends a famous theorem of Thomassen on $5$-choosability of planar graphs.

\begin{theorem}[Zhu \cite{ZhuAT}]\label{Theorem Zhu}
	Every planar graph $G$ satisfies $\AT (G)\leqslant 5$.
\end{theorem}

The proof has a similar structure to the elegant inductive argument of Thomassen, though an unexpected twist appears at the final stage. The following related results were recently obtained by a similar algebraic approach.

\begin{theorem}[Grytczuk and Zhu \cite{GrytczukZhu}]\label{Theorem Grytczuk and Zhu}
Every planar graph $G$ contains a matching $M$ such that $\AT (G-M)\leqslant 4$. In consequence, $G-M$ is $4$-choosable.
\end{theorem}

\begin{theorem}[Kim, Kim, and Zhu \cite{KimKimZhu}]\label{Theorem Kim, Kim, Zhu}
	Every planar graph $G$ contains a forest $F$ such that $\AT (G-F)\leqslant 3$. In consequence, $G-F$ is $3$-choosable.
\end{theorem}

Interestingly, the only known proofs of the above two choosability statements are algebraic, unlike it is for the theorem of Thomassen.

The following result was proved independently by Ellingham and Goddyn \cite{EllinghamGoddyn} (see also \cite{Alon Restricted}).

\begin{theorem}[Ellingham and Goddyn \cite{EllinghamGoddyn}]\label{Theorem AT Cubic Planar}
	Let $G$ be a $2$-connected cubic planar graph, and let $G_e$ denote the line graph of $G$. Then $\AT (G_e)\leqslant 3$. In particular, $G$ is $3$-edge choosable.
\end{theorem}

It is well known that $3$-colorability of cubic bridgeless planar graphs is equivalent to the Four Color Theorem. Unfortunately, the proof of the above result uses indirectly the validity of the Four Color Theorem in an essential way.

\section{The results}

Let $G$ be a simple connected plane graph whose faces are bounded by simple cycles. In this section we will present our results concerning various graphs derived from $G$, that is, graphs of the form $G_{\mathtt s}$ or $\overline{G}_{\mathtt s}$, where $\mathtt s$ denotes a string from the set $\{v,e,f,ve,vf,ef,vef\}$.

\subsection{Medial graphs} Recall that $\overline{G}_e$ is a graph obtained from the line graph of $G$ by restricting to facially incident edges. Another name for $\overline{G}_e$ is the \emph{medial graph} of $G$.

We will make use of the following result of Hladk\'{y}, Kr\'{a}l, and Schauz \cite{HladkyKS}, which constitutes an algebraic version of the famous theorem of Brooks \cite{Brooks}.

\begin{theorem}\label{Theorem AT Brooks} Let $G$ be a connected graph. If $G$ is neither a clique nor an odd cycle, then $\AT (G)\leqslant \Delta(G)$.
\end{theorem}
Since every medial graph is $4$-regular (and planar), we get immediately the following corollary extending a choosability result from \cite{FabriciJV}.

\begin{corollary}\label{Corollary Medial Brooks}
Every medial graph satisfies $\AT (\overline{G}_e)\leqslant 4$.
\end{corollary}

This can be improved when $G$ is a bipartite plane graph, as proved by Dross, Lu\v{z}ar, Macekov\'{a}, and Sot\'{a}k in \cite{DrossLMS}.

\begin{theorem}[Dross, Lu\v{z}ar, Macekov\'{a}, Sot\'{a}k \cite{DrossLMS}]\label{Theorem Medial Bipartite}
If $G$ is a plane bipartite graph, then $\AT (\overline{G}_e)\leqslant 3.$
\end{theorem}

Another improvement can be derived for triangulations directly from Theorem \ref{Theorem AT Cubic Planar}.

\begin{corollary}\label{Corollary Medial Triangulation}
Let $G$ be a plane triangulation. Then $\AT (\overline{G}_e)\leqslant 3$.
\end{corollary}
\begin{proof} If $G$ is a triangulation, then the dual graph $G^*$ is cubic and also bridgeless. Since the line graph of $G^*$ is the same as $\overline{G}_e$, the assertion follows from Theorem \ref{Theorem AT Cubic Planar}.
	\end{proof}

\subsection{Facial total graphs}

Fabrici, Jendrol', and Voigt proved in \cite{FabriciJV} that $\ch (\overline{G}_{ve})\leqslant 6$ for every plane graph $G$. We extend this result by proving the same bound for the Alon-Tarsi number of this graph.

\begin{theorem}\label{Theorem AT Gve}
Every plane graph $G$ satisfies $\AT (\overline{G}_{ve})\leqslant 6.$
\end{theorem}

\begin{proof}Let us denote $V=\{x_1,x_2,\dots,x_n\}$ and $E=\{y_1,y_2,\dots,y_m\}.$
	Since 
	$$\overline{G}_{ve}=G_v\cup \overline{G}_e \cup B_{ve},$$ we have $$P_{\overline{G}_{ve}}=P_{G_v}P_{\overline{G}_e}P_{B_{ve}}.$$
	By Theorem \ref{Theorem Zhu} the polynomial $P_{G_v}$ has a non-vanishing monomial $$M_v=x_1^{a_1}x_2^{a_2}\cdots x_n^{a_n},$$ with $a_i\leqslant 4$ for all $i=1,2,\dots,n$.
	By Corollary \ref{Corollary Medial Brooks} the polynomial $P_{\overline{G}_e}$ has a non-vanishing monomial
	$$M_e=y_1^{b_1}y_2^{b_2}\cdots y_m^{b_m},$$ with $b_i\leqslant 3$ for all $i=1,2,\dots,m$.
	Also, by Proposition \ref{Proposition Bipartite Polynomials} the polynomial $P_{B_{ve}}$ contains a non-vanishing monomial
	$$N=y_1^2y_2^2\cdots y_m^2.$$
	Indeed, we may orient all edges of $B_{ve}$ from $V$ to $E$, and each vertex in $E$ will have in-degree exactly $2$. Then the monomial corresponding to this orientation is precisely $N$.
	
	It follows that the polynomial $P_{\overline{G}_{ve}}$ contains a monomial $M=M_vM_eN$ in its multiset of monomials $\mathcal{M}_{\overline{G}_{ve}}$. This monomial $M$ has the form
	$$M=x_1^{a_1}\cdots x_n^{a_n}y_1^{b_1+2}\cdots y_m^{b_m+2}.$$
	Clearly, the maximum degree of $M$ is at most $5$.
	
It remains to show that $M$ is non-vanishing in $P_{\overline{G}_{ve}}$. To this end assume that $M=XYQ$, where $X$, $Y$, and $Q$ are non-vanishing monomials in the corresponding factors $P_{G_v}$, $P_{G_{e}}$, and $P_{B_{ve}}$ of $P_{\overline{G}_{ve}}$. First, notice that we must have $X=M_v$. Indeed, if some variable $x_i$ appears in $X$ with degree $\deg_{x_i}(X)<a_i$, then there must be another variable $x_j$ with $\deg_{x_j}(X)>a_j$ (since $P_{G_v}$ is a uniform polynomial). Thus, $\deg_{x_j}(XYQ)>a_j$, which excludes equality $M=XYQ$. In consequence, we must have $M_eN=YQ$.

Now, we claim that also $Q=N$. If not, then there must exist some variable $y_i$ for which $\deg_{y_i}(Q)<2$, as the maximum value of $\deg_{y_i}(P_{B_{ve}})$ is $2$. But then there must appear some $x_i$ in $Q$ (again, by the uniformity of a graph polynomial). However, there are no $x_i$'s in $M_eN$. So, we must have $Q=N$, which together with $X=M_v$, implies that also $Y=M_e$. This completes the proof.
	\end{proof}

By the same argument we may obtain improved upper bounds in some special cases.

\begin{theorem}\label{Theorem ve Bipartite Triangulation}If $G$ is a plane bipartite graph or a triangulation, then $\AT (\overline{G}_{ve})\leqslant 5$.
\end{theorem}
 
\begin{proof} By Theorem \ref{Theorem Medial Bipartite} and Corollary \ref{Corollary Medial Triangulation}, in both cases we have $\AT (\overline{G}_e)\leqslant 3$. This means that the graph polynomial $P_{\overline{G}_e}$ contains a non-vanishing monomial $M_e=y_1^{b_1}y_2^{b_2}\cdots y_m^{b_m},$ with all $b_i\leqslant 2$. Using the same notation as in the previous proof, we get that the maximum degree of a variable in the monomial $M=M_vM_eN$ is at most $4$. Also, by the same argumentation, $M$ is non-vanishing in $P_{\overline{G}_{ve}}$, which completes the proof.
	\end{proof}

In much the same way we may obtain analogous results for the edge-face coloring. It is enough to substitute the graph $G_v$ with the graph $G_f$.

\begin{theorem}\label{Theorem Edge-Face Coloring}
Every plane graph $G$ satisfies $\AT (\overline{G}_{ef})\leqslant 6$. If $G$ is a triangulation or a bipartite plane graph, then $\AT (\overline{G}_{ef})\leqslant 5$.
\end{theorem}

\subsection{Facial entire graphs} Fabrici, Jendrol', and Voigt proved in \cite{FabriciJV} that every plane graph $G$ satisfies $\ch (\overline{G}_{vef})\leqslant 8$. We prove the following strengthening of this result.

\begin{theorem}\label{Theorem AT Gvef}
Every plane graph $G$ satisfies $\AT (\overline{G}_{vef})\leqslant 8$. 
\end{theorem}

\begin{proof} Let us denote $V=\{x_1,x_2,\dots,x_n\}$, $E=\{y_1,y_2,\dots,y_m\}$, and $F=\{z_1,z_2,\dots,z_r\}$. Recall that the graph $\overline{G}_{vef}$ is a union of the following graphs: 
	$$\overline{G}_{vef}=G_v\cup \overline{G}_e\cup G_f\cup B_{ve}\cup B_{vf}\cup B_{ef}.$$ But we may also write it as 
	$$\overline{G}_{vef}=G_{vf}\cup \overline{G}_e\cup B_{ve}\cup B_{ef}.$$
	This implies that 
	$$P_{\overline{G}_{vef}}=P_{G_{vf}}P_{\overline{G}_e}P_{B_{ve}}P_{B_{ef}}.$$
	It is well-known that $G_{vf}$ is $7$-degenerate (see \cite{WangLih}), hence, by Proposition \ref{Proposition AT Degenerate}, $P_{G_{vf}}$ contains a non-vanishing monomial:
		$$M_{vf}=x_1^{a_1}x_2^{a_2}\cdots x_n^{a_n}z_1^{c_1}z_2^{c_2}\cdots z_r^{c_r},$$
		with $a_i\leqslant 7$ and $c_j\leqslant 7$ for all $i=1,2,\dots, n$, and $j=1,2,\dots, r$. By Corollary \ref{Corollary Medial Brooks} the polynomial $P_{\overline{G}_e}$ contains a non-vanishing monomial:
		$$M_e=y_1^{b_1}y_2^{b_2}\cdots y_m^{b_m},$$
		 with $b_i\leqslant 3$ for all $i=1,2,\dots m$. Finally, by Proposition \ref{Proposition Bipartite Polynomials} each of the two polynomials $P_{B_{ve}}$ and $P_{B{ef}}$ contains a non-vanishing monomial:
		 $$N=y_1^2y_2^2\cdots y_m^2.$$
		 Indeed, arguing similarly as in the last proof, we may orient all edges of $B_{ve}$ and $B_{ef}$ from $V$ to $E$ and from $F$ to $E$, respectively. Then each vertex in $E$ will have in-degree exactly $2$ in each of these orientations. The monomial corresponding to each of these orientations is precisely $N$.
		 
		 Putting all of this stuff together we get a monomial $M=M_{vf}M_eN^2$ in the multiset of monomials $\mathcal{M}_{\overline{G}_{vef}}$, which can be written as:
		 $$M=x_1^{a_1}x_2^{a_2}\cdots x_n^{a_n}y_1^{b_1+4}y_2^{b_2+4}\cdots y_m^{b_m+4}z_1^{c_1}z_2^{c_2}\cdots z_r^{c_r}.$$
		 
		 It remains to prove that $M$ is non-vanishing in the polynomial $P_{\overline{G}_{vef}}$. We shall argue similarly as in the proof of Theorem \ref{Theorem AT Gve}, namely we will demonstrate that the factorization $M=M_{vf}M_eN^2$ into four monomials from the corresponding factors of $P_{\overline{G}_{vef}}$ is unique. For this purpose, assume that $M=XYQR$, where $X$, $Y$, $Q$ and $R$ are monomials form the corresponding polynomials $P_{G_{vf}}$,$P_{\overline{G}_e}$,$P_{B_{ve}}$, and $P_{B_{ef}}$, respectively. As before, $X$ must be equal to $M_{vf}$ since otherwise some variable in $X$, either $x_i$ or $z_j$ will have degree strictly bigger than $a_i$ or $c_j$, respectively. This implies that $M_eN^2=YQR$.
		 
		 We claim now that both monomials, $Q$ and $R$, must be the same as $N$. First notice that the maximum degree that variable $y_i$ may have in $Q$ or $R$ is $2$. If $\deg_{y_i}(Q)<2$ or $\deg_{y_i}(R)<2$, then some $x_j$ or $z_k$ must appear in $Q$ or $R$, respectively. But then we have either $\deg_{x_i}(M)>a_i$ or $\deg_{z_j}(M)> c_j$. So, $Q=R=N$ must hold, and in consequence $M_e=Y$. This completes the proof.
	\end{proof}

As before, we may improve the above upper bound by one in the case of triangulations. We will need the following strengthening of the result of Ringel \cite{Ringel}.

\begin{theorem}\label{Theorem Triangulation G_vf}
If $G$ is a plane triangulation, then $\AT (G_{vf})\leqslant 6$.
\end{theorem}
\begin{proof}Let us denote $V=\{x_1,x_2,\dots,x_n\}$, and $F=\{z_1,z_2,\dots,z_r\}$. Recall that $G_{vf}$ is a union of three graphs:
	$$G_{vf}=G_v\cup G_{f}\cup B_{vf}.$$
So, we have also
$$P_{G_{vf}}=P_{G_v}P_{G_f}P_{B_{vf}}.$$
By Theorem \ref{Theorem Zhu} the polynomial $P_{G_v}$ has a non-vanishing monomial
$$M_v=x_1^{a_1}x_2^{a_2}\cdots x_n^{a_n},$$
with $a_i\leqslant 4$ for all $i=1,2,\dots,n$. By the assumption that $G$ is a triangulation, the graph $G_f$ is cubic and by Theorem \ref{Theorem AT Brooks} its polynomial $P_{G_f}$ contains a non-vanishing monomial
$$M_f=z_1^2z_2^2\cdots z_r^2,$$
unless $G=K_4$, but in this case $G_{vf}$ is $6$-regular and we may apply Theorem \ref{Theorem AT Brooks} directly. Notice that all degrees in the monomial $M_f$ are exactly quadratic because their sum must be equal to the number of edges in $G_f$. Finally, the polynomial $P_{B_{vf}}$ contains a non-vanishing monomial
$$N=z_1^3z_2^3\cdots z_r^3,$$ corresponding to the orientation of $B_{vf}$ with all edges directed from $V$ to $F$. Thus, it follows that there is a monomial $M=M_vM_fN$ in $P_{G_{vf}}$ which can be written as
$$M=x_1^{a_1}x_2^{a_2}\cdots x_n^{a_n}z_1^5z_2^5\cdots z_r^5.$$Arguing as in the two previous proofs, we may convince ourselves that $M$ is non-vanishing in $P_{G_{vf}}$.
	\end{proof}

We may now apply the above result to the following theorem.

\begin{theorem}\label{Theorem Triangulation G_vef}
If $G$ is a plane triangulation, then $\AT (\overline{G}_{vef})\leqslant 7$.
\end{theorem}
\begin{proof}The reasoning is exactly the same as in the proof of Theorem \ref{Theorem AT Gvef}. The only difference is that now the maximum degree of a variable in the monomial $M_{vf}$ is at most $6$ (by Theorem \ref{Theorem Triangulation G_vf}), while in $M_e$ it is at most $2$ (by Theorem \ref{Theorem AT Cubic Planar}). Therefore the maximum degree in $M$ is at most $6$, which completes the proof.
	\end{proof}

\section{Graph polynomials and paintability}

The following game theoretic variant of the list coloring problem was introduced independently by Schauz and Zhu. Let $G$ be simple graph with $n$ vertices and let $p:V(G) \rightarrow \mathbb{N}$ be a function. We say that $G$ is \emph{$p$-paintable} if the following conditions hold:
\begin{itemize}
	\item[(1)] for every $v \in V(G)$ we have $p(v) > 0;$
	\item[(2)] for every nonempty $X \subseteq V(G)$ there exists an independent
	subset $X' \subseteq X$ for which $G-X'$ is $q$-paintable, where for $v \in V(G)-X'$
	we define $q(v)=p(v)-1$ if $v\in X-X'$, and $q(v)=p(v)$, otherwise.
\end{itemize} 

Notice that the above definition always refers to a simpler instance of itself, because either $G-X'$ has fewer than $n$ vertices or (if $X'=\emptyset$) $$\sum_{v \in V(G)} q(v) < \sum_{v \in V(G)} p(v).$$ The base case of paintability is $V(G)= \emptyset$, while the base case of non-paintability is $p(v) = 0$ for at least one vertex $v\in V(G)$.

Notice also that the above definition corresponds to the painting game described in the Introduction. Indeed, in each round Lister reveals a new color in lists of the vertices belonging to $X$, and Painter chooses the subset $X' \subseteq X$ of the vertices that will be colored by this color. The winner is Painter if the set of uncolored vertices eventually becomes empty, while Lister wins the game if at some moment an uncolored vertex $v$ has $p(v)=0$, meaning that no new color 
will ever be available at $v$.

The least integer $k$ such that a graph $G$ is $p$-paintable with $p(v)\geqslant k$ is denoted by $\p (G)$ and called the \emph{painting number} of $G$. Clearly every graph satisfies $$\ch (G)\leqslant \p (G).$$

In \cite{Schauz} Schauz proved a surprising result relating the painting number to graph polynomials. We will give a refined, purely algebraic proof of his theorem for completeness.

\begin{theorem}[Schauz \cite{Schauz}]\label{Theorem Schauz}
	If the graph polynomial
	$$
	P_G(x_1,x_2,\ldots, x_n)=\prod_{x_ix_j\in E(G)}\left(x_i-x_j\right)
	$$  contains a non-vanishing monomial of multidegree $(\alpha_1,\alpha_2,\ldots, \alpha_n)$,
	then $G$ is  $p$-paintable for $p(x_i)=\alpha_i+1$. In particular, every graph $G$ satisfies $$\p (G)\leqslant\AT (G).$$
\end{theorem}

The theorem is a direct inductive consequence of the following Lemma, provided we assume $P_G \equiv 1$ for an edgeless graph $G$.

\begin{lemma}
	Let $X=\{x_1,x_2,\ldots,x_k\} \subseteq V(G)$ and $V(G)-X=\{y_1,y_2,\ldots,y_l\},$ where $k \geqslant 1$ and $l \geqslant 0.$ If $P_G$ contains a non-vanishing monomial of multidegree $(\alpha_1,\ldots, \alpha_k, \beta_1,\ldots,$ $ \beta_l),$ then there exists an independent subset  $X' \subseteq X,$ (we may assume that $X'=\{x_1,\ldots,x_j\}$, with $j \geqslant 0$), for which $P_{G-X'}$ contains a non-vanishing monomial of multidegree at most $(\alpha_{j+1}-1,\ldots, \alpha_k-1, \beta_1,\ldots, \beta_l).$
\end{lemma}

\begin{proof}
	
	Let us define projection operators of the following two types acting on the vector space 
	$\mathbb{R} [x_1,\ldots,x_k,y_1,\ldots,y_l]$:
	\begin{itemize}
		\item  $\Pi_{x_i}^{<\alpha_i}$, whose image is spanned by the polynomials $M$
		satisfying $\deg_{x_i}(M)<\alpha_i$, and kernel by those $N$
		for which $\deg_{x_i}(N) \geqslant \alpha_i$;
		\item $\Pi_{y_i}^{\beta_i}$, whose image is spanned by the polynomials $M$
		satisfying $\deg_{y_i}(M) = \beta_i$, and kernel by those $N$
		for which $\deg_{y_i}(N) \not= \beta_i$.
	\end{itemize}
	
	Let us now consider the following projection:
	$$
	T=(I-\Pi_{x_1}^{<\alpha_1}) \circ \ldots \circ (I-\Pi_{x_k}^{<\alpha_k})
	\circ\Pi_{y_1}^{\beta_1}\circ
	\ldots \circ\Pi_{y_l}^{\beta_l}.
	$$
	We easily see that in $T(P_G)$ only the monomials $M$
	satisfying $\deg_{x_i}(M) \geqslant \alpha_i$ and $\deg_{y_i}(M) = \beta_i$
	for all $i$ will remain unannihilated. But as $P_G$ is uniform, 
	there is only one such monomial, namely the one of multidegree $(\alpha_1,\ldots, \alpha_k,
	\beta_1,\ldots, \beta_l),$ which does not vanish by the assumption of the lemma. Consequently, substituting all values $x_i$ and $y_i$ equal to $1$ will yield a non-zero
	constant:
	$$
	(I-\Pi_{x_1}^{<\alpha_1}) \circ \ldots \circ (I-\Pi_{x_k}^{<\alpha_k})
	\circ\Pi_{y_1}^{\beta_1}\circ
	\ldots \circ\Pi_{y_l}^{\beta_l} (P_G)(1,\ldots,1) \in \mathbb{R} - \{0\}.
	$$
	
	On the other hand, the above expression is, by linearity, a combination
	of $2^k$ expressions of the form
	$$
	\Pi_{x_{j+1}}^{<\alpha_{j+1}} \circ \ldots \circ \Pi_{x_k}^{<\alpha_k}
	\circ\Pi_{y_1}^{\beta_1}\circ
	\ldots \circ\Pi_{y_l}^{\beta_l} (P_G)(1,\ldots,1),
	$$
	corresponding to all $2^k$ subsets of $\{1,\ldots,k\}$. At least one of them
	has therefore to be non-zero, hence at least one of
	$$
	\Pi_{x_{j+1}}^{<\alpha_{j+1}} \circ \ldots \circ \Pi_{x_k}^{<\alpha_k}
	\circ\Pi_{y_1}^{\beta_1}\circ
	\ldots \circ\Pi_{y_l}^{\beta_l} (P_G)
	$$
	has to be a non-zero polynomial, which means that $P_G$ has degrees
	at most $(\alpha_{j+1}-1,\ldots, \alpha_k-1)$ in the variables 
	$(x_{j+1},\ldots, x_k)$, exactly 
	$(\beta_1,\ldots, \beta_l)$ in $(y_1,\ldots, y_l)$,
	and arbitrary in $(x_1,\ldots,x_j)$.
	
	Let us now see that $X'=\{x_1, \ldots,x_j\}$ satisfies the conditions of the lemma:
	\begin{itemize}
		\item
		as $P_G$ has a monomial of degrees at most
		$(\alpha_{j+1}-1,\ldots, \alpha_k-1,
		\beta_1,\ldots, \beta_l)$ in the variables 
		$(x_{j+1},\ldots, x_k,y_1,\ldots, y_l)$,
		and $P_G$ can be seen as the product of
		$P_{G-X'}$ and some polynomial Q (which happens to be the product of all $u-v$ 
		representing the edges $uv \in E(G)$ with at least one end in $X'$), 
		by the very definition of polynomial multiplication we see that
		$P_{G-X'}$ also has a polynomial of degree at most
		$(\alpha_{j+1}-1,\ldots, \alpha_k-1, \beta_1,\ldots, \beta_l)$;
		\item the set $X'$ is independent: indeed, if $j \geqslant 2$ and for instance 
		$x_1x_2 \in E(G),$ then $P_G$ contains the factor $x_1-x_2$, which implies
		$P_G(x_1=1,x_2=1) \equiv 0,$ contradicting
		$$
		\Pi_{x_{j+1}}^{<\alpha_{j+1}} \circ \ldots \circ \Pi_{x_k}^{<\alpha_k}
		\circ\Pi_{y_1}^{\beta_1}\circ
		\ldots \circ\Pi_{y_l}^{\beta_l} (P_G)(1,\ldots,1) \in \mathbb{R} - \{0\},
		$$
		because the  operator of evaluation at $x_1=1$ commutes with both $\Pi_{x_i}^{<\alpha_i}$ 
		for $i \not=1$ and $\Pi_{y_i}^{\beta_i}$.
		
	\end{itemize}
	
	This completes the proof. 
\end{proof}

By the above theorem and our previous results, we may state now the following conclusion for paintability of diversely combined plane graphs.

\begin{theorem}\label{Theorem Paintability}
	If $G$ is a plane graph, then
	\begin{itemize}
		\item[(i)] $\p (\overline{G}_{ve})\leqslant 6$ and $\p(\overline{G}_{vef})\leqslant 8$. 
	\end{itemize}
If $G$ is a plane triangulation, then
\begin{itemize}
	\item [(ii)] $\p (\overline{G}_{ve})\leqslant 5$, $\p (\overline{G}_{vef})\leqslant 7$, and $\p (G_{vf})\leqslant 6$.
\end{itemize}
\end{theorem}

Perhaps some of these bounds could be obtained in a purely combinatorial way, but at present we do not see an easy modification of the existing methods avoiding the use of graph polynomials.
  
\section{Discussion}

We conclude the paper with some remarks and open problems. A natural question is whether the obtained bounds for the Alon-Tarsi number are optimal. For instance, in \cite{WangLih} Wang and Lih proved that $\ch (G_{vf})\leqslant 7$ holds for any plane graph $G$. It is not known, however, if this bound is optimal. In view of the decomposition 
$$G_{vf}=G_v\cup G_f \cup B_{vf},$$
and the results of Zhu (Theorem \ref{Theorem Zhu}) and Alon and Tarsi (Corollary \ref{Corollary AT Bipartite}), we know that there are non-vanishing monomials in the corresponding graph polynomials whose product is a monomial in $\mathcal{M}_{G_{vf}}$ with maximum variable degree at most $6$. If we could prove that it is non-vanishing, then we would obtain the following extension of the above result for $\ch (G_{vf})$.

\begin{conjecture}
	Every plane graph $G$ satisfies $\AT (G_{vf})\leqslant 7$.
\end{conjecture}

This conjecture, if true, would in turn imply the following improvement of Theorem \ref{Theorem AT Gvef}.

\begin{conjecture}
	Every plane graph $G$ satisfies $\AT (\overline{G}_{vef})\leqslant 7$.
\end{conjecture}

It is not known if the bound $\ch (\overline{G}_{vef})\leqslant 8$, established in \cite{FabriciJV} by Fabrici, Jendrol', and Voigt, is tight in general.

In principle, one may try to apply the polynomial method to any existing graph coloring problem. For instance, let us look at another famous challenge posed by Ringel \cite{Ringel EarthMoon} concerning the chromatic number of \emph{Earth-Moon graphs}, that is, graphs of the form $G=G_1\cup G_2$, where both subgraphs, $G_1$ and $G_2$, are planar. By the result of Zhu (Theorem \ref{Theorem Zhu}) we know that each of the graph polynomials, $P_{G_1}$ and $P_{G_2}$, contains a non-vanishing monomial of maximum degree at most $4$. Since $P_G=P_{G_1}P_{G_2}$, the graph polynomial $P_G$ contains a monomial with maximum degree at most $8$ (in the multiset of monomials $\mathcal{M}_G$). If we could prove that it is non-vanishing, then we would get the following statement.

\begin{conjecture}
	Every Earth-Moon graph $G$ satisfies $\AT (G)\leqslant 9$.
\end{conjecture} 

In 1973 Sulanke found examples of Earth-Moon graphs with $\chi(G)=9$ (see \cite{Gardner}), which gives currently best lower bound on the maximum chromatic number in this class. On the other hand, every Earth-Moon graph is $11$-degenerate and therefore satisfies $\chi(G)\leqslant12$, which is a currently best upper bound in this problem. However, by Proposition \ref{Proposition AT Degenerate} we also have $\AT (G)\leqslant 12$, so, any improvement of this bound would simultaneously make progress in the original Earth-Moon problem.

\end{document}